\documentclass[a4paper,11pt]{article}
\usepackage[latin1]{inputenc}
\usepackage[T1]{fontenc}
\usepackage{amsmath,amsthm,amsfonts,amssymb}
\usepackage{enumerate}
\usepackage[colorlinks]{hyperref}
\usepackage{color}
\usepackage{pst-node,pstricks-add,pst-text,pst-3d,pst-tree}

\newlength{\hdelta}
\newlength{\vdelta}
\theoremstyle{plain}
\newtheorem{theorem}{Theorem}[section]

\newtheorem*{theorem*}{Theorem}
\newtheorem*{lemma*}{Lemma}
\newtheorem*{proposition*}{Proposition}
\newtheorem*{conjecture*}{Conjecture}

\theoremstyle{definition}

\newtheorem*{question*}{Question}
\newtheorem*{example*}{Example}
\newtheorem*{remark*}{Remark}

\numberwithin{equation}{section}

\newcommand{\mailto}[1]{\href{mailto:#1}{\nolinkurl{#1}}}

\newcommand{\E}{\operatorname{E}}
\newcommand{\pr}{\operatorname{P}}

\newcommand{\eqst}{=_{\rm{st}}}
\newcommand{\lest}{\le_{\rm{st}}}

\newcommand{\Z}{\mathbb{Z}}

\newcommand{\R}{\mathbb{R}}

\begin{document}

\title{Flow coupling and stochastic ordering of throughputs in linear networks\thanks{This is an author-prepared preprint version of a conference article with the same title,
to appear in \emph{Proc.\ 8th International Conference on Performance Evaluation Methodologies and Tools (Valuetools 2014)}.}}
\author{
 Lasse Leskelä\thanks{
 Department of Mathematics and Systems Analysis,
 Aalto University School of Science,
 PO Box 11100,
 00076 Aalto, Finland.
 \quad URL: \url{http://math.aalto.fi/\~lleskela/} \quad
 Email: \protect\mailto{lasse.leskela@aalto.fi}}
}
\date{\today}
\maketitle

\maketitle
\begin{abstract}
Robust estimates for the performance of complicated queueing networks can be obtained by showing that
the number of jobs in the network is stochastically comparable to a simpler, analytically
tractable reference network. Classical coupling results on  stochastic ordering of network populations require strong monotonicity assumptions which are often violated in practice.
However, in most real-world applications we care more about what goes through a network than what
sits inside it. This paper describes a new approach for ordering \emph{flows
instead of populations} by augmenting network states with their associated flow counting
processes and deriving Markov couplings of the augmented
state--flow processes.
\end{abstract}

\newcommand{\CCrit}{R_{\rm cr}}
\newcommand{\DCritO}{D_{\rm cr}}
\newcommand{\DCritR}{D'_{\rm cr}}

\newcommand{\Xalt}{X^{\rm alt}}
\newcommand{\Salt}{S^{\rm alt}}
\newcommand{\Falt}{F^{\rm alt}}
\newcommand{\Xorig}{X^{\rm orig}}
\newcommand{\Sorig}{S^{\rm orig}}
\newcommand{\Forig}{F^{\rm orig}}

\newcommand{\nin}{n_{\rm in}}
\newcommand{\nout}{n_{\rm out}}
\newcommand{\net}{N}
\newcommand{\lambdain}{\lambda_{\rm in}}
\newcommand{\lambdaout}{\lambda_{\rm out}}

\newcommand{\sima}{\stackrel{a,a'}{\sim}}
\newcommand{\leGen}{\le^{\mathcal{A,B}}}
\newcommand{\leGena}{\le^{\mathcal{A,B}}_{a,a'}}
\newcommand{\leCC}{\le^{\rm cc}_{(a,a')}}
\newcommand{\leCCE}{\le^{\rm cc}_{(\tilde a, \tilde a')}}
\newcommand{\leLC}{\le^{\rm lc}_{(a,a')}}
\newcommand{\aggflow}[2]{f(#1,#2)}

\newcommand{\alphao}{\alpha^{\rm orig}}
\newcommand{\alphab}{\alpha^{\rm bal}}
\newcommand{\Fo}{F^{\rm orig}}
\newcommand{\Fb}{F^{\rm bal}}

\definecolor{Gray}{rgb}{0.3,0.3,0.3}
\definecolor{LightBlue}{HTML}{66CCFF}
\definecolor{DarkBlue}{HTML}{000066}
\definecolor{DarkBlueOld}{rgb}{0.1,0.1,0.5}
\definecolor{Purple}{rgb}{0.4,0.0,0.3}
\definecolor{NiceRed}{HTML}{CC0000}
\definecolor{PinkishRed}{rgb}{0.9,0.0,0.3}
\definecolor{verylightgray}{HTML}{F0F0F0}
\definecolor{DarkGray}{HTML}{080808}
\definecolor{JY-blue}{HTML}{0038A8}
\definecolor{JY-red}{HTML}{C85300}
\definecolor{Yellowish}{HTML}{FFFFCC}

\section{Introduction}

\subsection{Stochastic ordering of network flows}

Robust estimates for the performance of a complicated queueing network 
can be obtained by showing that the process $X$ describing
the number of jobs in the network is stochastically comparable to a simpler, analytically
tractable reference network $X'$. Classical coupling results on
stochastic ordering of $X$ and $X'$ require strong monotonicity assumptions
\cite{Economou_2003a,Leskela_2010,Massey_1987,Muller_Stoyan_2002,Whitt_1986} which are often violated in practice.
However, in most real-world applications we care more about what goes through a network than what
sits inside it. 
This paper describes a new alternative approach for ordering \emph{flows
instead of populations} by augmenting the network states ${X}$ and ${X'}$ with their associated flow-counting
processes ${F}$ and ${F'}$ and deriving Markov couplings of the augmented
state--flow processes $(X,F)$ and $(X',F')$ in an extended state space.

Earlier methods applicable for ordering of flows are mostly based on Markov reward comparison techniques (e.g.\ \cite{Boucherie_VanDijk_2009,Busic_Vliegen_Scheller-Wolf_2012,VanDijk_1998,VanDijk_VanDerWal_1989,VanHoutum_Zijm_Adan_Wessels_1998}). While more limited in scope than the general Markov reward approach, the flow coupling technique presented here, when applicable, yields stronger ordering results using simpler analysis. This paper will demonstrate this in the context of open linear queueing networks with general state-dependent arrival and service rates.

\subsection{Motivating example}
\label{sec:TandemQueue}
Consider a network of two queues in series where queue~1 and queue~2 have buffer capacities $s_1$
and $s_2$, respectively. Jobs arriving while queue~1 is full are rejected and lost, and the
server of queue~1 halts when queue~2 is full. When all interarrival times and job sizes are independent and exponentially distributed, the network population can be represented as a Markov jump process $\Xorig$ in the state space $\Sorig = \{x \in \Z_+^2: \, x_1 \le s_1, \, x_2 \le s_2\}$ with transitions
\begin{equation}
 \label{eq:Original}
 x \mapsto \left\{
 \begin{aligned}
  x + e_1            & \quad \text{at rate} \ \beta 1(x_1 < s_1), \\
  x - e_1 + e_2 & \quad \text{at rate} \  \delta_1(x_1) 1(x_2 < s_2), \\
  x - e_2            & \quad \text{at rate} \ \delta_2(x_2),
 \end{aligned}
 \right.
\end{equation}
where $\beta$ is the arrival rate of offered jobs, $\delta_i(x_i)$ is the service rate at queue~$i$ when
queue~$i$ has size $x_i$, $e_i$ is the $i$-th unit vector in $\Z^2$, and $1(A)$ is the indicator function which returns one if statement $A$ is true and zero otherwise. For example, when $\delta_i(x_i) = c_i x_i$, the system corresponds to a multi-server queue where all servers operate at rate $c_i$. When the system is
irreducible, the long-run mean loss rate and other equilibrium
statistical characteristics can in in principle be evaluated by solving a linear equation for the
equilibrium distribution. However, because the system is not reversible, solving the linear
problem analytically or numerically is hard for large $s_1,s_2$.

To obtain a computationally tractable upper bound for the equilibrium loss rate, van~Dijk and van~der~Wal
\cite{VanDijk_VanDerWal_1989} introduced a modification of the network dynamics so that arrivals
are blocked also in states where queue~2 is full, and the second server halts when queue~1 is
full. This so-called balanced system has a product-form equilibrium distribution and can be represented as a Markov jump process $\Xalt$ with
transitions
\begin{equation}
 \label{eq:UpperModel}
 x \mapsto \left\{
 \begin{aligned}
  x + e_1            & \quad \text{at rate} \ \beta 1(x_1 < s_1, x_2 < s_2), \\
  x - e_1 + e_2  & \quad \text{at rate} \  \delta_1(x_1) 1(x_2 < s_2), \\
  x - e_2            & \quad \text{at rate} \ \delta_2(x_2) 1(x_1<s_1).
 \end{aligned}
 \right.
\end{equation}
Because the state $(s_1,s_2)$ is transient for $\Xalt$, it is natural to define $\Xalt$ on the space
$\Salt = \Sorig \setminus \{(s_1,s_2)\}$.

The balanced service system $\Xalt$ employs a stricter admission policy and provides less service for
queue~2. Hence it is intuitively feasible to assume that the counting processes $\Forig_{\rm in}(t)$
and $\Falt_{\rm in}(t)$ describing the number of accepted jobs up to time $t$ are ordered
according to
\begin{equation}
 \label{eq:LossBounds}
 \Falt_{\rm in}(t) \le \Forig_{\rm in}(t)
\end{equation}
with respect to a suitable stochastic order. However, we are faced with the
following conceptual problem: If $\Xalt$ accepts less jobs, it should have shorter queues,
which should imply that $\Xalt$ spends less time in blocking states, and hence $\Xalt$ should accept
\emph{more jobs}. For this reason, a simple sample path argument cannot be used to
prove~\eqref{eq:LossBounds}.

Under the natural assumption that $\delta_1(x_1)$ and $\delta_2(x_2)$ are increasing\footnote{In this paper the terms `positive',
`increasing', and `less than' are synonyms for `nonnegative', `nondecreasing', and `less or equal
than', respectively.} functions, Van~Dijk and van~der~Wal \cite{VanDijk_VanDerWal_1989} proved that
\begin{equation}
 \label{eq:LossBoundsMean}
 \E \Falt_{\rm in}(t) \le \E \Forig_{\rm in}(t)
\end{equation}
by uniformizing the Markov processes into a discrete-time chain and applying inductive Markov reward comparison techniques. They also argued that a simpler coupling proof is not possible due to nonmonotone effects caused by
the blocking phenomena.

In this paper it will be shown that although neither of the above Markov processes is monotone with respect to the strong coordinatewise stochastic order, a strong coupling argument for proving \eqref{eq:LossBounds} is nevertheless possible. To accomplish this, this paper will introduce a Markov coupling in an extended space which carries redundant information about flow counting processes associated with the network population. As an application we obtain a simple proof of \eqref{eq:LossBounds} in the strong stochastic sense which at the same time strengthens \eqref{eq:LossBoundsMean} and  greatly simplifies its lengthy proof given in \cite{VanDijk_VanDerWal_1989}.

\section{Network population processes}

\subsection{Markov dynamics}

Consider a network consisting of a finite set of nodes $N=\{1,\dots,n\}$ where jobs randomly move across
directed links $L \subset (N \cup \{0\})^2$, and where node $0$ represents the outside world. The state of the system
is denoted by $x = (x_1,\dots,x_n) \in S$ where $S \subset \Z_+^N$ and $\Z_+$ denotes the positive integers.
The network dynamics is modeled as a Markov jump process $X = (X_1(t), \dots, X_n(t))_{t
\ge 0}$ in state space $S$ with transitions
\[
 x \mapsto x - e_i + e_j
 \quad \text{at rate} \ \alpha_{i,j}(x), \quad (i,j) \in L,
\]
where $e_i$ denotes the $i$-th unit vector in $\Z^n$, and $e_0$ stands as a synonym for zero. Here
\begin{itemize}
  \item $X_i(t)$ is the number of jobs in node $i$ at time $t$,
  \item $\alpha_{i,j}(x)$ for $i,j \in N$ is the instantaneous transition rate of jobs from node $i$ to node $j$ at state $x=(x_1,\dots,x_n)$,
  \item $\alpha_{0,i}(x)$ and $\alpha_{i,0}(x)$ are the arrival and departure rates of jobs for node $i$ at state $x=(x_1,\dots,x_n)$.
\end{itemize}
A collection of transition rates $\alpha_{i,j}: S \to \R_+$ and an initial state $X(0)$ defines such a Markov jump process in $S$, when $\alpha_{i,j}(x) = 0$ for all $x \in S$ such that $x-e_i+e_j \not\in S$ and satisfy the standard regularity condition  which guarantees that the Markov jump process is nonexplosive (see e.g.\ \cite{Bremaud_1999}).

\subsection{Augmented state--flow process}

The \emph{state--flow process} associated to population process $X$ generated by transition rates $\alpha_{i,j}$ is a Markov jump process $(X,F)$ taking values in $S
\times \Z_+^L$ and having transitions
\[
 (x,f) \mapsto (x - e_i + e_j, \ f + e_{i,j})
 \quad \text{at rate} \ \alpha_{i,j}(x),
 \quad (i,j) \in L,
\]
where $e_{i,j}$ denotes the unit vector in $\Z_+^L$ having its $(i,j)$-coordinate equal to one and other coordinates zero.
Here
\begin{itemize}
\item $X_i(t)$ is the number of jobs at node $i$ at time $t$
\item $F_{i,j}(t)-F_{i,j}(0)$ is the number of transitions across link $(i,j)$ during $(0,t]$.
\end{itemize}
This process is \emph{redundant} in that the second component of $(X,F)$ may be recovered from $F(0)$ and the
path of $X$ by the formula
\[
 F_{i,j}(t) - F_{i,j}(0) \ = \ \# \left\{ s \in (0,t]: \ X(s) - X(s-) = - e_i + e_j \right\},
\]
where $X(s-)$ denotes the left limit of $X$ at time $s$. Adding this redundancy allows to derive
useful non-Markov couplings of $X$ in terms of Markov couplings of $(X,F)$, as we shall soon see.

\subsection{Stochastic ordering and coupling}
\subsubsection{Strong stochastic order}
Let us recall some standard notations and facts about strong stochastic ordering of random processes.
For random vectors $A$ and $B$ in $\R^n$, we denote $A \lest B$ and say that $A$ is less than $B$ in the strong stochastic order if $\E \phi(A) \le \E \phi(B)$ for all $\phi: \R^n \to \R$ which are increasing with respect to the coordinatewise order on $\R^n$ and for which the expectations are defined. For real-valued random processes $(A_t)$ and $(B_t)$ indexed by a time parameter $t$ we denote $(A_t) \lest (B_t)$ if $(A_{t_1}, \dots, A_{t_n}) \lest (B_{t_1}, \dots, B_{t_n})$ for all finite collections of time parameters $(t_1,\dots,t_n)$.

Strong stochastic order allows to compare means in the sense that 
$A_t \lest B_t$ implies $\E A_t \le \E B_t$ whenever $A_t$ and $B_t$ are positive or have finite means.
Perhaps more importantly, it also allows to compare upper tail events in that $A_t \lest B_t$ always implies $\pr(A_t > s) \le \pr(B_t > s)$ for all real numbers $s$. In fact the latter property is equivalent to $A_t \lest B_t$, see e.g.\ \cite{Muller_Stoyan_2002,Shaked_Shanthikumar_2007}.

\subsubsection{Coupling}
For random vectors $A$ and $B$ we denote $A \eqst B$ if $A$ and $B$ have the same distribution. This definition is extended to random processes by denoting $(A_t) \eqst (B_t)$ if $(A_{t_1}, \dots, A_{t_n}) \eqst (B_{t_1}, \dots, B_{t_n})$ for all finite collections of time parameters $(t_1,\dots,t_n)$.

A bivariate random process $(\hat A_t, \hat B_t)$ indexed by time parameter $t$ is a \emph{coupling} of random processes $(A_t)$ and $(B_t)$ if $(\hat A_t) \eqst (A_t)$ and $(\hat B_t) \eqst (B_t)$. A simple computation using the definitions shows that if $(A_t)$ and $(B_t)$ admit a coupling which is ordered in the sense that $\hat A_t \le \hat B_t$ for all $t$ almost surely, then $(A_t) \lest (B_t)$. As a consequence of Strassen's coupling theorem, the converse implication is also true whenever the paths of $(A_t)$ and $(B_t)$ are right-continuous with left limits (e.g.\ \cite{Kamae_Krengel_OBrien_1977} or \cite[Thm 4.6]{Leskela_2010}).

\subsection{Marching soldiers coupling}
\label{sec:MarchingSoldiers}

\subsubsection{Coupling of population processes}
\label{sec:MarchingPopulations}
Fix a network with nodes $N=\{1,\dots,n\}$ and directed links $L \subset (N \cup \{0\})^2$, and consider
two population processes $X$ in $S \subset \Z_+^N$ and $X'$ in $S' \subset \Z_+^N$, generated by state-dependent transition rates $\alpha_{i,j}(x)$ and $\alpha'_{i,j}(x)$, respectively. In most applications the state spaces $S$ and $S'$ are assumed to be identical, but this restriction is not needed for the results developed in this paper.

A natural and simple way to couple two Markov population processes of the above type is to force both processes to locally take identical steps with as high rate as possible. This so-called \emph{marching soldiers coupling}~\cite{Chen_2005} of the Markov jump processes $X$ and $X'$ is defined as a Markov jump process $(\hat X, \hat X')$ in $S \times S'$ having the transitions
\begin{align*}
 &(x,x') \mapsto \\
 & \quad \left\{
 \begin{aligned}
  (x - e_i + e_j, \, x'-e_i+e_j)  & \quad \text{at rate} \ {\alpha_{i,j}}(x) \wedge {\alpha'_{i,j}}(x'), \\
  (x, \, x'-e_i+e_j)         & \quad \text{at rate} \ ({\alpha'_{i,j}}(x')   - {\alpha_{i,j}}(x))_+, \\
  (x - e_i + e_j, \, x')         & \quad \text{at rate} \ ({\alpha_{i,j}}(x)     - {\alpha'_{i,j}}(x'))_+,
 \end{aligned}
 \right.
\end{align*}
for $(i,j) \in L$. Here we use the shorthands $a \wedge a' = \min\{a,a'\}$ and $a_+ = \max\{a,0\}$.
By inspecting the marginal transition rates for each state $x \in S$ and $x' \in S'$, one can check that the process $\hat X$ (resp.\ $\hat X'$) is a Markov jump process by itself and has the same transition rates as $X$ (resp.\ $X'$). That is, $(\hat X, \hat X')$ is a Markov coupling of $X$ and $X'$.

\subsubsection{Coupling of state--flow processes}
\label{sec:MarchingFloex}

The marching soldiers coupling of state--flow processes $(X,F)$ and $(X',F')$ associated to population processes $X$ and $X'$ is defined analogously as a Markov jump process $(\hat X, \hat F, \hat X', \hat F')$ in the state space $(S \times \Z_+^L) \times (S' \times \Z_+^L)$ having the transitions
\begin{align*}
 &((x,f),(x',f')) \mapsto \\
 & \quad \left\{
 \begin{aligned}
  (T_{i,j}(x,f), T_{i,j}(x',f'))  & \quad \text{at rate} \ {\alpha_{i,j}}(x) \wedge {\alpha'_{i,j}}(x'), \\
  ((x,f), T_{i,j}(x',f'))         & \quad \text{at rate} \ ({\alpha'_{i,j}}(x')   - {\alpha_{i,j}}(x))_+, \\
  (T_{i,j}(x,f), (x',f'))         & \quad \text{at rate} \ ({\alpha_{i,j}}(x)     - {\alpha'_{i,j}}(x'))_+, \\
 \end{aligned}
 \right.
\end{align*}
for $(i,j) \in L$, where $T_{i,j}(x,f) = (x-e_i+e_j, f + e_{i,j})$ denotes the extended state obtained from state $(x,f)$ after moving one job from node~$i$ to node~$j$.

\subsection{Flow balance}

If $(X,F)$ is the state--flow process associated to a population process $X$, then
\begin{multline*}
 X_i(t) - X_i(0) \ =  \ \!\!\! \sum_{j: (j,i) \in L} \!\!\! (F_{j,i}(t)-F_{j,i}(0)) \ \\
 - \!\!\! \sum_{j: (i,j) \in L} \!\!\! (F_{i,j}(t)-F_{i,j}(0))
\end{multline*}
for all $t \ge 0$. This flow conservation equality shows that the quantity
\[
 X_i(t) - \!\!\! \sum_{j: (j,i) \in L} \!\!\! F_{j,i}(t) \
 +\!\!\! \sum_{j: (i,j) \in L} \!\!\! F_{i,j}(t)
\]
remains constant over time for all nodes $i$. As a consequence, any coupling
$(\hat X, \hat F, \hat X', \hat F')$ of state--flow processes $(X,F)$ and $(X',F')$ automatically preserves the relation
\begin{equation}
 \label{eq:NodeBalance}
 x_i - \!\!\! \sum_{j: (j,i) \in L} \!\!\! f_{j,i} \ + \!\!\! \sum_{j: (i,j) \in L} \!\!\! f_{i,j}
 \ \ = \ \
 x'_i - \!\!\! \sum_{j: (j,i) \in L} \!\!\! f'_{j,i} \ + \!\!\! \sum_{j: (i,j) \in L} \!\!\! f'_{i,j},
\end{equation}
for all nodes $i$, in the sense that the set of pairs $(x,f)$ and $(x',f')$ related according to~\eqref{eq:NodeBalance}
is absorbing for $(\hat X, \hat F, \hat X', \hat F')$.

\section{Open linear networks}

\subsection{Linear network dynamics}

Consider an open linear network of $n$ nodes represented by a directed graph $(N \cup \{0\},L)$ with node set $N=\{1,\dots,n\}$ and
link set $L = \{(0,1),(1,2),\dots,(n-1,n),(n,0)\}$, see Figure~\ref{fig:OpenNetwork}.

\begin{figure}[h]
\begin{center}
  \psset{unit=0.58cm}
  \begin{pspicture}(0,-0.5)(8,1.5)
  \psset{fillstyle=none,linecolor=white}
  \Cnode(-3.5,0){N0}
  \Cnode(11.5,0){N4}

  \psset{fillstyle=solid,fillcolor=LightBlue,linecolor=black}
  \cnodeput(0,0){N1}{1}
  \cnodeput(4,0){N2}{2}
  \cnodeput(8,0){N3}{3}

  \psset{linecolor=black,nodesep=3pt,arrows=->,arrowsize=5pt}
  \ncline{N0}{N1}\naput{$\lambda(x)$}
  \ncline{N1}{N2}\naput{$\mu_1(x)$}
  \ncline{N2}{N3}\naput{$\mu_2(x)$}
  \ncline{N3}{N4}\naput{$\mu_3(x)$}
  \end{pspicture}
\end{center}
\caption{\label{fig:OpenNetwork} Open linear network with $n=3$ nodes.}
\end{figure}
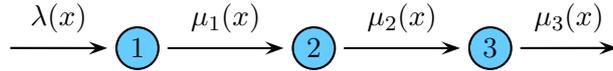

The system state is a vector $x=(x_1,\dots,x_n)$ with values in a state space $S \subset \Z_+^N$.
We model the network dynamics by a Markov jump process $X(t) \in S$ with transitions
\begin{equation}
 \label{eq:NOriginalN}
 x \mapsto \left\{
 \begin{aligned}
  x + e_1           & \quad \text{at rate} \ \lambda(x), \\
  x - e_i + e_{i+1} & \quad \text{at rate} \ \mu_i(x), \quad i=1,...,n-1, \\
  x - e_n           & \quad \text{at rate} \ \mu_n(x).
 \end{aligned}
 \right.
\end{equation}
In queueing terminology, $\lambda(x)$ is the arrival rate of jobs to node~1, and $\mu_i(x)$ can be thought of as the aggregate service rate at node~$i$.
The fact that $\lambda(x)$ and $\mu_i(x)$ may depend on the full network state $x=(x_1,\dots,x_n)$ allows to model queueing systems with admission control in front, and blocking and interference effects among the servers.
We assume that $\lambda(x) = 0$ for all $x \in S$ such that $x+e_1 \notin S$ and $\mu_i(x) = 0$ for all $x \in S$ such that $x-e_i+e_{i+1} \notin S$.

The above model is a special instance of a general population process defined in the previous section, where $\alpha_{0,1}(x) = \lambda(x)$ and $\alpha_{i,i+1}(x) = \mu_i(x)$ for $i=1,\dots,n$.

\subsection{Strong ordering of flows}

Consider now two population processes $X$ and $X'$ on the open linear network of $n$ nodes generated by
state-dependent transition rates $(\lambda,\mu_1,\dots,\mu_n)$ and $(\lambda',\mu_1',\dots,\mu_n')$, respectively. Here the state spaces $S$ of $X$ and $S'$ of $X'$ are subsets of $\Z_+^N$ which
in most applications are identical, but for the following result this restriction is not needed.

The following is the main result of this paper. It gives a sufficient condition for the strong stochastic ordering of flow counting processes associated to $X$ and $X'$. For convenience, we define $i+1 := 0$ for $i=n$ below.

\begin{theorem}
\label{the:LinearNetwork}
Assume that the following implications hold for all $x \in S$ and $x' \in S$ and for all $i=1,\dots,n-1$:
\begin{align}
 \label{eq:NRateToFirst}
 x_1 \ge x_1'                                     &\implies \lambda(x) \le \lambda'(x'), \\
 \label{eq:NRateFromI}
 x_i \le x_i' \ \text{and} \ x_{i+1} \ge x_{i+1}' &\implies \mu_i(x) \le \mu_i'(x'), \\
 \label{eq:NRateFromLast}
 x_n \le x_n'                                     &\implies \mu_n(x) \le \mu_n'(x').
\end{align}
Then the associated flow counting processes initiated at zero are ordered according to
\begin{equation}
 \label{eq:LinearOrderSt}
 ({F_{i,i+1}}(t))_{t \ge 0} \ \lest \ ({F'_{i,i+1}}(t))_{t \ge 0}
\end{equation}
for all $i = 0,1,\dots,n$, whenever ${X}(0) \eqst {X'}(0)$.
\end{theorem}

\begin{proof}
It is sufficient to construct a coupling of the state--flow processes $(X,F)$ and $(X',F')$ for which~\eqref{eq:LinearOrderSt} holds for all $i$ and all $t \ge 0$ with probability one.
Let
\[
 (\hat X, \hat F, \hat X', \hat F')
\]
be a marching soldiers coupling of $(X,F)$ and $(X',F')$ started at the (possibly random) initial state
\[
 (X(0),F(0),X(0),F(0)),
\]
as defined in Sec.~\ref{sec:MarchingSoldiers}.
By assumption, $F(0)=F'(0)$ is the zero vector in $\Z_+^L$. Note that the above vector couples the initial states of $(X,F)$ and $(X',F')$
because $X(0)$ and $X'(0)$ have the same distribution. The latter assumption also implies that $X(0)$ takes its values in $S \cap S'$ almost surely.

We define a relation between state--flow pairs $(x,f) \in S \times \Z_+^L$ and $(x',f') \in S' \times \Z_+^L$ by denoting $(x,f) \sim (x',f')$ if
\begin{equation}
 \label{eq:FlowOrder}
 f_{i,i+1} \le f'_{i,i+1}
\end{equation}
for all $i=0,1,\dots,n$ and if \eqref{eq:NodeBalance} holds for all $i=1,\dots,n$. We note that $(\hat X(0), \hat F(0)) \sim (\hat X'(0), \hat F'(0))$ almost surely. To finish the proof it suffices to show that the set of state--flow pairs that are related according to $\sim$ is an absorbing set for the marching soldiers coupling. Note that both sides of \eqref{eq:NodeBalance} are invariant to any possible transition of the processes. Hence we only need to show that none of the inequalities \eqref{eq:FlowOrder} can ever be broken by any transition of the coupled process.

Let us first show that \eqref{eq:FlowOrder} cannot be broken for $i=0$. Consider a state--flow pair related according to $(x,f) \sim (x',f')$. If $f_{0,1} < f'_{0,1}$, then a single transition cannot break the inequality $f_{0,1} \le f'_{0,1}$. Consider next the case where $f_{0,1} = f'_{0,1}$. Then the flow conservation equality \eqref{eq:NodeBalance} at node $1$ implies that
\[
 x_1 - x_1' = f_{1,2}' - f_{1,2} \ge 0.
\]
Thus $x_1 \ge x'_1$ which in light of \eqref{eq:NRateToFirst} implies that $\lambda(x) \le \lambda'(x')$. This shows that the marching soldiers coupling has zero transition rate for the transition $((x,f),(x',f')) \mapsto ((x+e_1,f+e_{0,1}),(x',f'))$. But this is the only transition which potentially could break \eqref{eq:FlowOrder} for $i=0$.

Let us next show that \eqref{eq:FlowOrder} cannot be broken for $1 \le i \le n-1$. Consider a state--flow pair related according to $(x,f) \sim (x',f')$. Again, we only need to study the case where $f_{i,i+1} = f'_{i,i+1}$. Then the flow conservation equality \eqref{eq:NodeBalance} at node $i$ implies that
\[
 x_i' - x_i = f_{i-1,i}' - f_{i-1,i} \ge 0,
\]
whereas \eqref{eq:NodeBalance} for node $i+1$ implies that
\[
 x_{i+1}' - x_{i+1} = f_{i+1,i+2} - f_{i+1,i+2}' \le 0.
\]
Thus $x_i \le x'_i$ and $x_{i+1} \ge x_{i+1}'$ which in light of \eqref{eq:NRateFromI} imply that $\mu_i(x) \le \mu_i'(x')$. This shows that the marching soldiers coupling has zero rate for the transition $((x,f),(x',f')) \mapsto ((x-e_i+e_{i+1},f+e_{i,i+1}),(x',f'))$. But this is the only transition which potentially could break \eqref{eq:FlowOrder} for $i$.

Let us finally show that \eqref{eq:FlowOrder} cannot be broken for $i=n$. Consider a state--flow pair related according to $(x,f) \sim (x',f')$. Again, we only need to consider next that case where $f_{n,0} = f'_{n,0}$. Then the flow conservation equality \eqref{eq:NodeBalance} at node $n$ implies that
\[
 x_n' - x_n = f_{n-1,n}' - f_{n-1,n} \ge 0.
\]
Thus $x_n \le x'_n$ which in light of \eqref{eq:NRateFromLast} implies that $\mu_n(x) \le \mu_n'(x')$. This shows that the marching soldiers coupling has zero transition rate for the transition $((x,f),(x',f')) \mapsto ((x+e_n,f+e_{n,0}),(x',f'))$. But this is the only transition which potentially could break  \eqref{eq:FlowOrder} for $i$.

Because the marching soldiers coupling may never exit the set of ordered state--flow pairs, we conclude that $(\hat X(t), \hat F(t)) \sim (\hat X'(t), \hat F'(t))$ for all $t \ge 0$, and especially $\hat F_{i,i+1}(t) \le \hat F'_{i,i+1}(t)$ for all $i$ and all $t \ge 0$ almost surely.
\end{proof}

\subsection{Strong ordering of populations}

To understand how Theorem~\ref{the:LinearNetwork} is structurally different from more well-known ordering and coupling results for Markov population processes, let $X$ and $X'$ as in the previous section. The following result gives a sufficient condition for the strong stochastic ordering of the population processes $X$ and $X'$ with respect to the coordinatewise order on $\R^n$. For vectors in $\R^n$ we write $(x_1,\dots,x_n) \le (x'_1,\dots,x'_n)$ if $x_i \le x'_i$ for all $i$. For convenience, we define $i+1 := 0$ for $i=n$ below.

\begin{theorem}
\label{the:LinearPopulation}
Assume that the following implications hold for all $x \in S$ and $x' \in S'$ such that $x \le x'$ and for all $i=2,\dots,n$:
\begin{align}
 \label{eq:PopRateToFirst}
 x_1 = x_1'                                     &\implies \lambda(x) \le \lambda'(x') \ \text{and} \ \mu_1(x) \ge \mu'_1(x'), \\
 \label{eq:PopRateFromI}
 x_i = x_i'                     &\implies  \mu_{i-1}(x) \le \mu'_{i-1}(x') \ \text{and} \ \mu_i(x) \ge \mu'_i(x').
\end{align}
Then $(X(t)) \lest (X'(t))$ whenever $X(0) \lest X'(0)$.
\end{theorem}
\begin{proof}
Let $(\hat X(0), \hat X'(0))$ be a coupling of $X(0)$ and $X'(0)$ such that $\hat X(0) \le \hat X'(0)$ with probability one. Such a coupling exists by Strassen's coupling theorem (e.g.\ \cite{Muller_Stoyan_2002,Shaked_Shanthikumar_2007}).

Let $(\hat X, \hat X')$ be a marching soldiers coupling of $X$ and $X'$ as described in Sec.~\ref{sec:MarchingPopulations}, started at the initial state $(\hat X(0), \hat X'(0))$. We will show that the marching soldiers coupling never exits the set of state pairs ordered according to the coordinatewise order, that is, the set $\{(x,x') \in S \times S': x \le x'\}$ is absorbing for the Markov process $(\hat X, \hat X')$.

Consider a pair of states such that $x \le x'$, and let us try to break the ordering $x_1 \le x'_1$. This is possible in a single transition only if $x_1 = x_1'$, in which case \eqref{eq:PopRateToFirst} implies that $\lambda(x) \le \lambda'(x')$ and  $\mu_1(x) \ge \mu'_1(x')$. But then the transitions $(x,x') \mapsto (x+e_1,x')$ and $(x,x') \mapsto (x,x'-e_1)$ both have zero rate for the marching soldiers coupling. These are the only transitions for the marching coupling which could break the relation $x_1 \le x_1'$.

Consider next a pair of states such that $x \le x'$, and let us try to break the ordering $x_i \le x'_i$ for some $i \ge 2$. This is possible in a single transition only if $x_i = x_i'$, in which case \eqref{eq:PopRateFromI} implies that $\mu_{i-1}(x) \le \mu'_{i-1}(x')$ and $ \mu_i(x) \ge \mu'_i(x')$. But then the transitions $(x,x') \mapsto (x-e_{i-1}+e_i,x')$ and $(x,x') \mapsto (x,x'-e_i+e_{i+1})$ both have zero rate for the marching soldiers coupling. These are the only transitions for the marching coupling which could break the relation $x_i \le x_i'$.

We conclude that $\hat X(t) \le \hat X'(t)$ for all $t \ge 0$ almost surely, and therefore the claim follows.
\end{proof}

Note that Theorem~\ref{the:LinearPopulation} can also be proved as a consequence of a generic relation preservation result in~\cite[Example 5.7]{Leskela_2010} (see alternatively \cite{Lopez_Sanz_2002}), or by applying the transition rate conditions in \cite{Massey_1987,Whitt_1986}.

\section{Application: Throughput ordering in a tandem queue}

Let us now revisit the tandem queueing system of Sec.~\ref{sec:TandemQueue}. The balanced model described by \eqref{eq:UpperModel} corresponds to a population process $X$ on a 2-node linear network where $S = \Salt$ and
\begin{align*}
 \lambda(x_1,x_2) &= \beta 1(x_1<s_1,x_2<s_2),\\
 \mu_1(x_1,x_2) &= \delta_1(x_1) 1(x_2 < s_2),\\
 \mu_2(x_1,x_2) &= \delta_2(x_2) 1(x_1 < s_1).
\end{align*}
The original model described by \eqref{eq:Original} corresponds to a similar population process $X'$ where $S' = \Sorig$ and
\begin{align*}
 \lambda'(x_1,x_2) &= \beta 1(x_1<s_1),\\
 \mu'_1(x_1,x_2) &= \delta_1(x_1) 1(x_2 < s_2),\\
 \mu'_2(x_1,x_2) &= \delta_2(x_2).
\end{align*}
In this case Theorem~\ref{the:LinearPopulation} cannot be applied to order populations according to $X(t) \lest X'(t)$ because condition~\eqref{eq:PopRateFromI} fails for $i=2$ due to $\mu_2(x) < \mu'_2(x')$ when $x_1 = x_1' = s_1$ and $0<x_2=x_2' < s_2$. Neither can Theorem~\ref{the:LinearPopulation} cannot be applied to order populations according to $X'(t) \lest X(t)$ because condition~\eqref{eq:PopRateToFirst} fails due to $\lambda'(x') > \lambda(x)$ when $x'_1 = x_1 < s_1$ and $x'_2 < x_2 = s_2$.

Nevertheless, the augmented state--flow process can be coupled with the help of Theorem~\ref{the:LinearNetwork}.
Indeed, the conditions of Theorem~\ref{the:LinearNetwork} are valid if and only if the service rates $\delta_1(x_1)$ and $\delta_2(x_2)$ are increasing. When this natural monotonicity condition is satisfied and when both systems are started at the same initial state and zero flow counters, Theorem~\ref{the:LinearNetwork} implies that all flow counting processes are ordered by
\begin{align*}
 F_{0,1}(t) &\lest F'_{0,1}(t),\\
 F_{1,2}(t) &\lest F'_{1,2}(t),\\
 F_{2,0}(t) &\lest F'_{2,0}(t).
\end{align*}
Because $F_{0,1}(t)$ and $F'_{0,1}(t)$ are the counting processes of accepted jobs in the balanced system and the original system, we obtain \eqref{eq:LossBounds}.

\section{Conclusions}

This paper discussed the strong stochastic ordering and coupling of network populations and their flow counting processes. Easily verifiable sufficient conditions were given for the transition rates of population processes on open linear networks which imply that the associated flow counting processes can be ordered using a natural coupling in the augmented space of state--flow processes. Important open problems include (i) to study into what extent the given sufficient conditions are also necessary and (ii) to extend the analysis into networks with two-way flows and more general network topologies. These problems are subjects of ongoing research.

\subsubsection*{Acknowledgements}
This work has been financially supported by the Emil Aaltonen Foundation and the Academy of Finland.

\bibliographystyle{abbrv}
\bibliography{lslReferences}
\end{document}